\pdfoutput=1
\documentclass[11pt]{article}

\usepackage[margin=1in]{geometry}
\usepackage{amsmath,amssymb,amsthm,mathtools}
\usepackage{enumitem}
\usepackage{hyperref}
\usepackage{microtype}

\hypersetup{
  colorlinks=true,
  linkcolor=blue,
  citecolor=blue,
  urlcolor=blue
}

\newtheorem{theorem}{Theorem}[section]

\newtheorem{corollary}[theorem]{Corollary}

\theoremstyle{remark}
\newtheorem{remark}[theorem]{Remark}

\newcommand{\C}{\mathbb{C}}
\newcommand{\R}{\mathbb{R}}
\newcommand{\E}{\mathbb{E}}
\newcommand{\Et}{\widetilde{\mathbb{E}}}
\newcommand{\ip}[2]{\left\langle #1,#2\right\rangle}
\newcommand{\Tr}{\operatorname{Tr}}
\newcommand{\Herm}{\operatorname{Herm}}

\title{Degree-Four Vector-Coordinate SoS Cannot Detect the MUB Upper Bound}
\author{Shreyhaan Sarkar}
\date{\today}

\begin{document}
\maketitle

\begin{abstract}
We prove a degree-four Sum-of-Squares lower bound for the natural vector-coordinate formulation of mutually unbiased bases. For every dimension $d$ and every proposed number $m$ of bases, independent Haar-random orthonormal bases define a degree-four pseudoexpectation satisfying the orthonormality constraints and the cross-unbiasedness constraints in the quartic equality formulation. The same pseudoexpectation satisfies the degree-four localizing constraints for the $2\times 2$ Hermitian semidefinite formulation of the cross-coherence inequalities. Consequently, degree-four vector-coordinate Sum-of-Squares cannot refute the existence of $m$ mutually unbiased bases in $\C^d$, even when $m>d+1$ and genuine MUBs are impossible. We then show that this failure is formulation-dependent: in a projector-coordinate formulation, degree-four SoS recovers the standard upper bound $m\le d+1$. Thus degree-four vector-coordinate SoS and degree-four projector-coordinate SoS are separated already by the elementary MUB counting obstruction. Applied to Randomstrasse101 Open Problem 23, this gives a negative answer for the two vector-coordinate encodings explicitly described there, while leaving stronger lifted/projector formulations as a separate question.
\end{abstract}

\begin{center}
\end{center}

\section{Introduction}

A collection of orthonormal bases
\[
B_1,\ldots,B_m\subset \C^d
\]
is a collection of \emph{mutually unbiased bases} if, for every two distinct bases $B_k,B_\ell$ and every $v\in B_k,w\in B_\ell$,
\[
|\ip{v}{w}|=\frac{1}{\sqrt d}.
\]
Equivalently, $|\ip{v}{w}|^2=1/d$. It is known that $m\le d+1$, with equality achieved when $d$ is a prime power; see, for example, \cite{BBRV2002}. The first unresolved dimension is $d=6$, where it is conjectured that there do not exist seven mutually unbiased bases.

Randomstrasse101 Open Problem 23 asks whether there is a degree-four Sum-of-Squares proof that no seven mutually unbiased bases exist in $\C^6$ \cite{Randomstrasse2026}. The problem statement is explicitly phrased using $7\times 6$ vector variables $v_i^{(k)}\in\C^6$. It then says that the constraints may be encoded by quartic relations on the real and imaginary parts of the vector components, and it also gives an alternative $2\times2$ Hermitian semidefinite encoding of the cross-basis inequalities. This note studies exactly these two vector-coordinate encodings.

The main result is that degree four is intrinsically too weak for these vector-coordinate formulations. More generally, for every $d$ and every $m$, the standard vector-coordinate MUB system admits a degree-four pseudoexpectation. Thus degree-four vector-coordinate SoS cannot even recover the elementary upper bound $m\le d+1$.

The construction is simple. Take $m$ independent Haar-random orthonormal bases of $\C^d$. These bases are not mutually unbiased pointwise, but they satisfy the relevant constraints up to degree four. Same-basis orthonormality holds exactly, and if $x,y\in\C^d$ are independent Haar-random unit vectors, then
\[
\E |\ip{x}{y}|^2=\frac1d.
\]
At degree four, a quartic cross-unbiasedness equation can only be multiplied by a constant, so this gives a pseudoexpectation for the quartic equality formulation. For the $2\times2$ semidefinite formulation, we prove a reusable affine multiplier estimate: multiplication by $\ip{x}{y}$ has norm at most $1/d$ between the relevant affine-linear spaces.

We also include a formulation-separation result. If one instead uses projector variables $P=vv^*$, then the cross-unbiasedness condition becomes quadratic:
\[
\Tr(P_i^{(k)}P_j^{(\ell)})=\frac1d.
\]
Degree four in the projector variables is therefore a strictly stronger relaxation. In fact, a degree-four projector-coordinate SoS argument recovers the standard upper bound $m\le d+1$. Hence the failure of degree-four SoS proved here is not a universal failure of all formulations; it is a failure of the direct vector-coordinate formulation.

This formulation sensitivity is consistent with the broader optimization literature on MUBs. Kolountzakis, Matolcsi, and Weiner developed a positive-definite-functions approach to MUB upper bounds \cite{KMW2018}; Bandeira, Doppelbauer, and Kunisky analyzed limitations of that approach at low degree \cite{BDK2022}; and Gribling and Polak studied noncommutative polynomial-optimization hierarchies with symmetry reduction \cite{GriblingPolak2024}. The present paper is deliberately more elementary and formulation-specific: it isolates a degree-four obstruction for the direct vector-coordinate SoS relaxation.

\section{The exact Randomstrasse101 formulation}

The Randomstrasse101 entry defines two orthonormal bases $\{v_1,\ldots,v_d\}$ and $\{w_1,\ldots,w_d\}$ in $\C^d$ to be mutually unbiased when
\[
|\ip{v_i}{w_j}|=\frac1{\sqrt d}
\]
for all $i,j$ \cite{Randomstrasse2026}. It recalls that $\operatorname{MUB}(d)\le d+1$ and that equality is achieved in prime-power dimensions, then states the dimension-six conjecture $\operatorname{MUB}(6)<7$.

Open Problem 23 asks:
\begin{quote}
Is there a Sum-of-Squares degree 4 proof that there are no 7 Mutually Unbiased Bases in $\C^6$?
\end{quote}
The entry then explains that this is a statement about the existence of vectors
\[
v_i^{(k)}\in\C^6,
\qquad i=1,\ldots,6,
\qquad k=1,\ldots,7,
\]
satisfying the prescribed overlap pattern. The same paragraph says that these constraints can be encoded by quartic relations on the real and imaginary parts of the vector components, using
\[
(v^*u)(u^*v)=|\ip{u}{v}|^2.
\]
It then gives a second encoding: the same-basis constraints are quadratic, and the cross-basis conditions may be written as
\[
|\ip{v_i^{(k)}}{v_j^{(\ell)}}|\le \frac1{\sqrt6}
\]
for $k\ne \ell$, which are semidefinite constraints on $2\times2$ Hermitian matrices involving the quadratic quantity $\ip{v_i^{(k)}}{v_j^{(\ell)}}$.

\begin{remark}[Squared-overlap normalization]
In the displayed squared-overlap equation in the Randomstrasse101 PDF, the cross-basis value appears typographically as $1/\sqrt6$ inside an expression that is already squared. The definition immediately above, and the subsequent $2\times2$ LMI threshold, use the standard MUB convention $|\ip{v}{w}|=1/\sqrt6$. Thus throughout this paper the quartic equality is written as
\[
|\ip{v}{w}|^2=\frac16
\]
in dimension six, and as $|\ip{v}{w}|^2=1/d$ in dimension $d$.
\end{remark}

Thus the results below map directly onto the two vector-coordinate models described in the problem statement: the quartic equality model and the $2\times2$ Hermitian LMI model. We do not claim to settle a different model in which the variables are projectors $P=vv^*$ or noncommutative projector symbols.

\section{Vector-coordinate MUB systems}

Fix integers $d\ge2$ and $m\ge1$. For $k=1,\ldots,m$ and $i=1,\ldots,d$, let
\[
v_i^{(k)}\in\C^d
\]
be vector variables. We write
\[
v_i^{(k)}=(v_{i,1}^{(k)},\ldots,v_{i,d}^{(k)}).
\]
All polynomial statements below may be interpreted over the real coordinate ring generated by the real and imaginary parts of these coordinates. Complex notation is used for readability.

The same-basis orthonormality constraints are
\[
\ip{v_i^{(k)}}{v_j^{(k)}}=\delta_{ij},
\]
or equivalently, in real coordinates,
\[
\operatorname{Re}\ip{v_i^{(k)}}{v_j^{(k)}}=\delta_{ij},
\qquad
\operatorname{Im}\ip{v_i^{(k)}}{v_j^{(k)}}=0.
\]
These are quadratic equations.

For $k\ne \ell$, the quartic equality formulation imposes
\[
|\ip{v_i^{(k)}}{v_j^{(\ell)}}|^2=\frac1d.
\]
Equivalently,
\[
q_{ij}^{k\ell}:=
|\ip{v_i^{(k)}}{v_j^{(\ell)}}|^2-\frac1d=0.
\]
These are quartic equations in the real vector coordinates.

The semidefinite formulation instead imposes, for $k\ne\ell$,
\[
M_{ij}^{k\ell}:=
\begin{pmatrix}
1/\sqrt d & \ip{v_i^{(k)}}{v_j^{(\ell)}}\\
\ip{v_j^{(\ell)}}{v_i^{(k)}} & 1/\sqrt d
\end{pmatrix}
\succeq0.
\]
Together with same-basis orthonormality, these inequalities are classically equivalent to mutual unbiasedness. Indeed, if $B_\ell$ is an orthonormal basis and $x$ is a unit vector, then
\[
\sum_{j=1}^d |\ip{x}{v_j^{(\ell)}}|^2=1.
\]
If each summand is at most $1/d$, then every summand equals $1/d$.

\section{Degree-four pseudoexpectations and refutations}

A degree-four pseudoexpectation is a real linear functional
\[
\Et:\R[x]_{\le4}\to\R
\]
on real polynomials of degree at most four such that
\[
\Et[1]=1
\]
and
\[
\Et[f^2]\ge0
\]
for every real polynomial $f$ of degree at most two.

If equality constraints $q_s=0$ are present, we say that $\Et$ satisfies them to degree four if
\[
\Et[hq_s]=0
\]
whenever $\deg(hq_s)\le4$.

If a Hermitian matrix polynomial $M(x)\succeq0$ of degree at most two is present, the degree-four localizing condition is
\[
\Et[g(x)^*M(x)g(x)]\ge0
\]
for every vector polynomial $g$ whose entries have degree at most one. It is enough to verify this after complexifying the real polynomial space, because complex positivity restricts to real positivity.

A degree-four SoS refutation of an equality system has the form
\[
-1=\sum_r f_r^2+\sum_s h_s q_s,
\]
where $\deg(f_r)\le2$ and every term $h_s q_s$ has degree at most four. For a mixed equality and semidefinite system, a degree-four refutation additionally allows terms
\[
g_t^*M_tg_t
\]
with $g_t$ affine-linear. A degree-four pseudoexpectation satisfying all degree-four constraints rules out such a refutation, since applying $\Et$ to the identity gives $-1\ge0$.

\section{The Haar pseudoexpectation}

Let
\[
U_1,\ldots,U_m
\]
be independent Haar-random unitary matrices in $U(d)$. Define
\[
v_i^{(k)}:=U_ke_i.
\]
Thus $v_1^{(k)},\ldots,v_d^{(k)}$ are the columns of $U_k$ and form an orthonormal basis of $\C^d$.

For every real polynomial $p$ of degree at most four in the vector coordinates, define
\[
\Et[p]:=\E[p(v)].
\]
This is an actual expectation, hence
\[
\Et[1]=1,
\qquad
\Et[f^2]\ge0
\]
for every real polynomial $f$ of degree at most two.

Same-basis constraints hold pointwise. For every $k$,
\[
\ip{v_i^{(k)}}{v_j^{(k)}}=\delta_{ij}
\]
almost surely. Therefore, if $q$ is any same-basis orthonormality constraint and $\deg(hq)\le4$, then
\[
\Et[hq]=0.
\]

If $k\ne\ell$, then $v_i^{(k)}$ and $v_j^{(\ell)}$ are independent Haar-random unit vectors in $\C^d$. Therefore
\[
\E |\ip{v_i^{(k)}}{v_j^{(\ell)}}|^2=\frac1d.
\]
Indeed, by unitary invariance, fix $v_i^{(k)}=e_1$. Then
\[
|\ip{v_i^{(k)}}{v_j^{(\ell)}}|^2=|(v_j^{(\ell)})_1|^2,
\]
and by symmetry of a Haar-random unit vector,
\[
\E |(v_j^{(\ell)})_1|^2=\frac1d.
\]
Thus
\[
\Et\left[
|\ip{v_i^{(k)}}{v_j^{(\ell)}}|^2-\frac1d
\right]=0.
\]

Since this cross constraint has degree four in the vector coordinates, a degree-four SoS refutation can multiply it only by constants. Therefore the Haar pseudoexpectation satisfies all quartic equality constraints to degree four.

\begin{theorem}[Quartic vector-coordinate lower bound]\label{thm:quartic}
For every $d\ge2$ and $m\ge1$, the standard vector-coordinate quartic equality formulation of $m$ mutually unbiased bases in $\C^d$ admits a degree-four pseudoexpectation. Consequently, this formulation has no degree-four Sum-of-Squares refutation.
\end{theorem}

\begin{proof}
The pseudoexpectation is the Haar pseudoexpectation constructed above. Positivity follows because it is an actual expectation. Same-basis constraints vanish pointwise. Cross-basis quartic constraints have zero pseudoexpectation, and since they already have degree four, only constant multipliers are possible in a degree-four refutation. Applying the pseudoexpectation to any alleged degree-four refutation gives
\[
-1=\sum_r \Et[f_r^2]\ge0,
\]
which is impossible.
\end{proof}

\section{An affine Haar multiplier theorem}

We now isolate the estimate needed for the $2\times2$ Hermitian LMI formulation. Let $x,y$ be independent Haar-random unit vectors in $\C^d$, and set
\[
z=\ip{x}{y}.
\]
The theorem says that multiplication by $z$ couples affine-linear functions only weakly: its operator norm is at most $1/d$.

\begin{theorem}[Affine Haar multiplier bound]\label{thm:affine-multiplier}
Let $x,y$ be independent Haar-random unit vectors in $\C^d$, and let $z=\ip{x}{y}$. Let $\mathcal A$ be the complex vector space of affine-linear polynomials in the vector-coordinate variables of any finite collection of independent Haar-random orthonormal bases containing $x$ and $y$. Then for all $p,q\in\mathcal A$,
\[
\left|\E[\overline p\,zq]\right|
\le
\frac1d
\sqrt{\E|p|^2\,\E|q|^2}.
\]
Equivalently, the bilinear form
\[
(p,q)\longmapsto \E[\overline p\,\ip{x}{y}\,q]
\]
has operator norm at most $1/d$ with respect to the $L^2$ norms on affine-linear functions.
\end{theorem}

\begin{proof}
We use the convention
\[
\ip{x}{y}=\sum_{r=1}^d \overline{x_r}y_r.
\]
The joint law of the Haar bases is invariant under independent phase rotations of each column. In particular, it is invariant under
\[
x\mapsto e^{i\theta}x,
\qquad
 y\mapsto e^{i\phi}y,
\]
and under independent phase rotations of all other columns.

Complexify the affine-linear space. Any affine-linear polynomial is a sum of a constant, linear coordinate terms, and conjugate-linear coordinate terms. In
\[
\E[\overline p\,zq],
\]
a monomial contribution can be nonzero only if its total phase is neutral under every independent column phase.

The phase bookkeeping is as follows; each row denotes any coordinate of the indicated vector:
\[
\begin{array}{c|c}
\text{factor} & \text{phase under } x\mapsto e^{i\theta}x,\; y\mapsto e^{i\phi}y \\
\hline
x & e^{i\theta}\\
\overline{x} & e^{-i\theta}\\
y & e^{i\phi}\\
\overline{y} & e^{-i\phi}\\
z=\ip{x}{y} & e^{-i\theta}e^{i\phi}
\end{array}
\]
A monomial contribution in $\E[\overline p\,zq]$ survives only if its total phase is trivial under both independent phase rotations. If a term of $p$ has phase exponent $\alpha\in\mathbb Z^2$ and a term of $q$ has phase exponent $\beta\in\mathbb Z^2$, then the corresponding term in $\overline p\,zq$ has exponent
\[
-\alpha+(-1,1)+\beta.
\]
Thus phase balance requires
\[
\beta=\alpha+(1,-1).
\]
Because $p$ and $q$ are affine-linear, the only possible matches are
\[
\alpha=(0,1),\quad \beta=(1,0),
\]
which gives
\[
p=a\cdot y,
\qquad
q=b\cdot x,
\]
and
\[
\alpha=(-1,0),\quad \beta=(0,-1),
\]
which gives
\[
p=a\cdot \overline x,
\qquad
q=b\cdot \overline y,
\]
for vectors $a,b\in\C^d$. Constant terms, all terms involving other columns, and all other $x,y$-linear or conjugate-linear combinations have unbalanced phase and integrate to zero.

For the first block, write
\[
p=a\cdot y=\sum_{r=1}^d a_r y_r,
\qquad
q=b\cdot x=\sum_{t=1}^d b_t x_t.
\]
Then
\[
\overline p=\sum_{r=1}^d \overline{a_r}\,\overline{y_r},
\qquad
z=\sum_{s=1}^d \overline{x_s}y_s,
\qquad
q=\sum_{t=1}^d b_t x_t.
\]
Using independence of $x$ and $y$,
\[
\begin{aligned}
\E[\overline p\,zq]
&=
\sum_{r,s,t}
\overline{a_r}b_t
\E[\overline{y_r}y_s]
\E[\overline{x_s}x_t].
\end{aligned}
\]
For a Haar-random unit vector $u\in\C^d$,
\[
\E[\overline{u_r}u_s]=\frac{\delta_{rs}}d.
\]
Therefore
\[
\E[\overline p\,zq]
=
\frac1{d^2}\sum_{r=1}^d \overline{a_r}b_r.
\]
Also,
\[
\E|p|^2=\frac{\|a\|^2}{d},
\qquad
\E|q|^2=\frac{\|b\|^2}{d}.
\]
Thus
\[
\left|\E[\overline p\,zq]\right|
\le
\frac{\|a\|\|b\|}{d^2}
=
\frac1d\sqrt{\E|p|^2\,\E|q|^2}.
\]

The second nonzero block, $p=a\cdot\overline x$ and $q=b\cdot\overline y$, gives the same estimate. The two blocks are orthogonal in $L^2$ by phase invariance. More explicitly, if $p_y$ and $p_{\overline x}$ denote the projections of $p$ onto the spans of the coordinates of $y$ and of $\overline x$, respectively, and $q_x,q_{\overline y}$ are defined similarly, then
\[
\E[\overline p\,zq]
=
\frac1{d^2}
\left(\langle a_y,b_x\rangle+\langle a_{\overline x},b_{\overline y}\rangle\right),
\]
whereas
\[
\E|p|^2\ge \frac{\|a_y\|^2+\|a_{\overline x}\|^2}{d},
\qquad
\E|q|^2\ge \frac{\|b_x\|^2+\|b_{\overline y}\|^2}{d}.
\]
Cauchy's inequality on the direct sum of the two relevant blocks gives the claimed bound.
\end{proof}

\section{\texorpdfstring{The $2\times2$ Hermitian semidefinite formulation}{The 2x2 Hermitian semidefinite formulation}}

Fix a cross pair
\[
x=v_i^{(k)},
\qquad
y=v_j^{(\ell)},
\qquad
k\ne\ell.
\]
Then $x,y$ are independent Haar-random unit vectors. Put
\[
z=\ip{x}{y},
\qquad
c=\frac1{\sqrt d}.
\]
The $2\times2$ LMI is
\[
M(x,y)=
\begin{pmatrix}
c & z\\
\overline z & c
\end{pmatrix}
\succeq0.
\]
The degree-four localizing condition requires that, for all affine-linear complex polynomials $p,q$,
\[
\Et\left[
c|p|^2+c|q|^2+\overline p\,zq+\overline q\,\overline z\,p
\right]\ge0.
\]

\begin{theorem}[LMI vector-coordinate lower bound]\label{thm:lmi}
For every $d\ge2$ and $m\ge1$, the Haar pseudoexpectation satisfies the degree-four localizing constraints for the $2\times2$ Hermitian semidefinite vector-coordinate formulation of $m$ mutually unbiased bases in $\C^d$. Consequently, this formulation has no degree-four Sum-of-Squares refutation.
\end{theorem}

\begin{proof}
Fix a cross pair $x,y$, and let $z=\ip{x}{y}$ and $c=1/\sqrt d$. For affine-linear $p,q$, put
\[
A=\Et|p|^2,
\qquad
B=\Et|q|^2.
\]
By Theorem~\ref{thm:affine-multiplier},
\[
\left|\Et[\overline p\,zq]\right|
\le
\frac1d\sqrt{AB}.
\]
Therefore
\[
\begin{aligned}
&\Et\left[
c|p|^2+c|q|^2+\overline p\,zq+\overline q\,\overline z\,p
\right]\\
&=cA+cB+2\operatorname{Re}\Et[\overline p\,zq]\\
&\ge cA+cB-\frac2d\sqrt{AB}.
\end{aligned}
\]
Since $A+B\ge2\sqrt{AB}$ and $c=1/\sqrt d\ge1/d$, the last expression is nonnegative. Hence the localizing condition holds.

If a degree-four SoS refutation of the semidefinite formulation existed, applying the Haar pseudoexpectation would make the square terms nonnegative, the same-basis equality terms vanish pointwise, and the LMI localizing terms nonnegative by the preceding paragraph. The right side would therefore have pseudoexpectation at least zero, while the left side has pseudoexpectation $-1$, a contradiction.
\end{proof}

\section{A separation from projector-coordinate degree four}

The preceding results show that degree-four vector-coordinate SoS cannot see the MUB upper bound. This is not true for natural projector-coordinate formulations. In this section we make the comparison precise by isolating a centered Hilbert-Schmidt Gram subsystem that is implied by rank-one projector MUB variables.

Let
\[
P_i^{(k)}=v_i^{(k)}v_i^{(k)*}
\]
be rank-one Hermitian projectors. In projector variables, the cross-unbiasedness condition is
\[
\Tr(P_i^{(k)}P_j^{(\ell)})=\frac1d,
\qquad k\ne\ell,
\]
which is quadratic in the entries of $P$. Define the centered projectors
\[
Q_i^{(k)}=P_i^{(k)}-\frac Id.
\]
Then $Q_i^{(k)}$ belongs to the real Hilbert space $\Herm_0(d)$ of traceless Hermitian $d\times d$ matrices, whose dimension is
\[
D=d^2-1.
\]
With the Hilbert-Schmidt inner product, the MUB projector relations imply
\[
\Tr(Q_i^{(k)}Q_j^{(\ell)})=
\begin{cases}
\delta_{ij}-1/d, & k=\ell,\\
0, & k\ne\ell.
\end{cases}
\]

We now define the centered projector-coordinate Gram system used below. Choose an orthonormal real basis of $\Herm_0(d)$ and introduce variables
\[
q_i^{(k)}\in\R^D,
\qquad i=1,\ldots,d,
\qquad k=1,\ldots,m.
\]
The constraints are the quadratic Gram equations
\begin{equation}\label{eq:centered-projector-gram}
\ip{q_i^{(k)}}{q_j^{(\ell)}}=
\begin{cases}
\delta_{ij}-1/d, & k=\ell,\\
0, & k\ne\ell.
\end{cases}
\end{equation}
These are precisely the centered Hilbert-Schmidt Gram constraints implied by rank-one projector MUB variables. The theorem below uses only the constraints \eqref{eq:centered-projector-gram}. Hence it applies to any stronger projector-coordinate formulation in which these centered Gram relations are included, or in which they are derived as degree-two consequences of the imposed projector constraints.

\begin{theorem}[Projector-coordinate degree four recovers the MUB upper bound]\label{thm:projector-bound}
In the centered projector-coordinate Gram system \eqref{eq:centered-projector-gram}, degree-four SoS proves that $m\le d+1$. More precisely, if $m>d+1$, the constraints \eqref{eq:centered-projector-gram} have a degree-four SoS refutation.
\end{theorem}

\begin{proof}
Write the variables as $q_a\in\R^D$, where $a$ ranges over all pairs $(k,i)$. Define the $D\times D$ matrix
\[
S=\sum_a q_aq_a^T.
\]
The following identity is a degree-four sum of squares in the coordinates of the $q_a$'s:
\[
D\Tr(S^2)-(\Tr S)^2
=
\sum_{r<s}(S_{rr}-S_{ss})^2+2D\sum_{r<s}S_{rs}^2\ge0.
\]
Here each entry $S_{rs}$ is quadratic in the $q$-variables, so the displayed identity has degree four.

We reduce this SoS polynomial modulo the quadratic constraints \eqref{eq:centered-projector-gram}, allowing multipliers of degree at most two. First,
\[
\Tr S=\sum_a \|q_a\|^2
\equiv md\left(1-\frac1d\right)=m(d-1).
\]
Second,
\[
\Tr(S^2)=\sum_{a,b}\ip{q_a}{q_b}^2.
\]
For a fixed basis $k$,
\[
\sum_{i,j=1}^d (\delta_{ij}-1/d)^2=d-1,
\]
and the cross-basis inner products are constrained to be zero. Hence
\[
\Tr(S^2)\equiv m(d-1).
\]
These congruences are degree-four congruences: for example,
\[
\ip{q_a}{q_b}^2-c_{ab}^2=(\ip{q_a}{q_b}+c_{ab})(\ip{q_a}{q_b}-c_{ab}),
\]
where $c_{ab}$ denotes the corresponding right-hand side of \eqref{eq:centered-projector-gram}, and the multiplier has degree two.

Therefore
\[
D\Tr(S^2)-(\Tr S)^2
\equiv Dm(d-1)-m^2(d-1)^2
=m(d-1)^2(d+1-m)
\]
modulo the degree-four ideal generated by \eqref{eq:centered-projector-gram}. If $m>d+1$, this constant is negative. Let
\[
c=m(d-1)^2(m-d-1)>0.
\]
The preceding congruence says that the degree-four SoS polynomial $F:=D\Tr(S^2)-(\Tr S)^2$ satisfies
\[
F\equiv -c
\]
modulo the constraints. Equivalently, for the quadratic Gram constraints $g_s$ in \eqref{eq:centered-projector-gram}, there are polynomial multipliers $h_s$ of degree at most two such that
\[
F+c=\sum_s h_s g_s.
\]
Hence
\[
-1=\frac1c F-\frac1c\sum_s h_s g_s.
\]
This is a degree-four SoS refutation, since equality-constraint multiples may appear with arbitrary signs.
\end{proof}

\begin{corollary}[Formulation separation]\label{cor:separation}
For every $d\ge2$ and every $m>d+1$, degree-four vector-coordinate SoS admits the Haar pseudoexpectation and hence cannot refute $m$ proposed MUBs in $\C^d$, while degree-four centered projector-coordinate Gram SoS refutes the same instance. In particular, degree-four vector-coordinate SoS and degree-four projector-coordinate SoS are separated by the basic MUB upper bound.
\end{corollary}

\begin{remark}
For a concrete instance, take $d=2$ and $m=4$. Four MUBs in $\C^2$ are impossible because $m\le3$. The degree-four vector-coordinate relaxation nevertheless has the Haar pseudoexpectation. By Theorem~\ref{thm:projector-bound}, the centered projector-coordinate degree-four relaxation refutes the instance. Similarly, in the dimension-six setting, the projector-coordinate argument refutes $m=8$ proposed bases but is exactly tight at $m=7$.
\end{remark}

\section{Consequences for dimension six and Open Problem 23}

Taking $d=6$ and $m=7$ in Theorems~\ref{thm:quartic} and~\ref{thm:lmi} gives the following consequences.

\begin{corollary}
In the vector-coordinate quartic equality formulation described in Randomstrasse101 Open Problem 23, there is no degree-four Sum-of-Squares proof that seven mutually unbiased bases do not exist in $\C^6$.
\end{corollary}

\begin{corollary}
In the vector-coordinate $2\times2$ Hermitian semidefinite formulation described in Randomstrasse101 Open Problem 23, there is no degree-four Sum-of-Squares proof that seven mutually unbiased bases do not exist in $\C^6$.
\end{corollary}

Thus, for the two vector-coordinate encodings explicitly described in Open Problem 23, the answer is negative. The proof is stronger than the dimension-six case: degree-four vector-coordinate SoS cannot refute the existence of $m$ mutually unbiased bases in $\C^d$ for any $m$, even for $m>d+1$.

However, Theorem~\ref{thm:projector-bound} shows that this negative result is not a blanket statement about all degree-four formulations. Projector variables already allow degree-four SoS to recover the general upper bound $m\le d+1$. The dimension-six problem at $m=7$ is exactly the tight case of that general bound, so the projector-coordinate separation does not settle whether a stronger projector-variable degree-four proof can rule out seven MUBs in $\C^6$.

\section{Scope and limitations}

The conclusion of this paper is formulation-dependent. It applies to the vector-coordinate formulations in which the variables are the real and imaginary parts of the vectors $v_i^{(k)}$. These are the two formulations described explicitly in Randomstrasse101 Open Problem 23: quartic real-coordinate equations and $2\times2$ Hermitian semidefinite constraints involving the quadratic inner products.

The result does not rule out degree-four Sum-of-Squares refutations in projector variables $P_i^{(k)}=v_i^{(k)}v_i^{(k)*}$ or in noncommutative/projector-algebra hierarchies. In projector variables, cross-unbiasedness is quadratic, so degree four can test squared deviations of cross overlaps. The Haar pseudoexpectation used here does not satisfy those stronger degree-four projector-variable constraints, because for independent Haar unit vectors $x,y$,
\[
|\ip{x}{y}|^2
\]
has expectation $1/d$ but nonzero variance.

Therefore the correct interpretation is:
\begin{quote}
The natural vector-coordinate degree-four SoS relaxation is intrinsically too weak to detect even the basic MUB counting obstruction, whereas projector-coordinate degree four is already strong enough to recover that obstruction.
\end{quote}

This separation suggests that any successful low-degree SoS approach to the dimension-six MUB problem must use a stronger lifted formulation than the direct vector-coordinate one.

\end{document}